\documentclass[11pt]{article}
\usepackage{amsmath}
\usepackage{amsfonts}
\usepackage{amssymb}
\usepackage{amsthm}
\usepackage{palatino}
\usepackage[margin=2.5cm, vmargin={1.5cm},includefoot]{geometry}
\usepackage{color,epsfig,latexsym,graphicx,eucal,layout} 

\title{Transformation laws for generalized Dedekind sums associated to Fuchsian groups}

\author{Claire Burrin, Jay Jorgenson,
Cormac O'Sullivan and Lejla Smajlovi\'c\footnote{{\it Date:} April 25, 2019.
\newline \indent \ \ \
  \newline \indent \ \ \
The second and third authors were partially supported by PSC-CUNY grants.}
}
\date{}

\renewcommand\footnotemark{}

\begin{document}

\maketitle

\def\s#1#2{\langle \,#1 , #2 \,\rangle}

\def\H{{\mathbb H}}
\def\F{{\mathfrak F}}
\def\C{{\mathbb C}}
\def\R{{\mathbb R}}
\def\Z{{\mathbb Z}}
\def\Q{{\mathbb Q}}
\def\N{{\mathbb N}}
\def\st{{\mathbb S}}
\def\D{{\mathbb D}}
\def\B{{\mathbb B}}
\def\G{{\Gamma}}
\def\GH{{\G \backslash \H}}
\def\g{{\gamma}}
\def\L{{\Lambda}}
\def\ee{{\varepsilon}}
\def\K{{\mathcal K}}
\def\Re{\text{\rm Re}}
\def\Im{\text{\rm Im}}
\def\SL{\text{\rm SL}}
\def\GL{\text{\rm GL}}
\def\PSL{\text{\rm PSL}}
\def\sgn{\text{\rm sgn}}
\def\tr{\text{\rm tr}}
\def\F{\mathcal{F}}
\def\ca{{\mathfrak a}}
\def\cb{{\mathfrak b}}
\def\cc{{\mathfrak c}}
\def\cd{{\mathfrak d}}
\def\ci{{\infty}}

\def\sa{{\sigma_\mathfrak a}}
\def\sb{{\sigma_\mathfrak b}}
\def\sc{{\sigma_\mathfrak c}}
\def\sd{{\sigma_\mathfrak d}}
\def\si{{\sigma_\infty}}

\def\se{{\sigma_\eta}}
\def\sz{{\sigma_{z_0}}}

\def\sai{{\sigma^{-1}_\mathfrak a}}
\def\sbi{{\sigma^{-1}_\mathfrak b}}
\def\sci{{\sigma^{-1}_\mathfrak c}}
\def\sdi{{\sigma^{-1}_\mathfrak d}}
\def\sii{{\sigma^{-1}_\infty}}
\def\PSL{\text{\rm PSL}}
\def\vol{\text{\rm vol}}
\def\I{\text{\rm Im}}

\newcommand{\m}[4]{\begin{pmatrix}#1&#2\\#3&#4\end{pmatrix}}
\newcommand{\n}{\frac1{\sqrt{37}}}
\newcommand{\ms}[4]{\left(\smallmatrix #1&#2\\#3&#4\endsmallmatrix\right)}
\newcommand{\ns}{\textstyle\frac1{\sqrt{37}}}
\newcommand{\e}{\eqref}


\newtheorem{theorem}{Theorem}[section]
\newtheorem{lemma}[theorem]{Lemma}
\newtheorem{prop}[theorem]{Proposition}
\newtheorem{cor}[theorem]{Corollary}
\newtheorem{conj}[theorem]{Conjecture}
\newtheorem{remark}{Remark}
\newtheorem{defs}[theorem]{Definition}
\renewcommand{\labelenumi}{(\roman{enumi})}

\numberwithin{equation}{section}

\let\originalleft\left
\let\originalright\right
\renewcommand{\left}{\mathopen{}\mathclose\bgroup\originalleft}
\renewcommand{\right}{\aftergroup\egroup\originalright}

\bibliographystyle{plain}

\begin{abstract}\noindent
We establish transformation laws for generalized Dedekind sums associated to the Kronecker limit function of non-holomorphic  Eisenstein series
and their higher-order variants. These results apply to general Fuchsian groups of the first kind, and examples are provided in the cases
of the Hecke triangle groups, the Hecke congruence groups $\Gamma_0(N)$, and the non-congruence arithmetic groups $\Gamma_0(N)^+$.
\end{abstract}

\section{Introduction} 

\subsection{The reciprocity law for the Dedekind sum}

The Dedekind eta function is defined by the infinite product
\begin{equation}\label{etad}
\eta(z):= e^{\pi i z/12} \prod_{n=1}^{\infty}\left(1-e^{2\pi i n z}\right)
\end{equation}
with $z$ in the upper half plane $\H$.  When $(\smallmatrix a & b \\ c & d \endsmallmatrix ) \in \SL(2, \Z)$
it can be shown that 
\begin{equation}\label{log_eta}
\log \eta\left(\frac{az+b}{cz+d}\right) = \log \eta(z) + \frac{1}{2}\log(cz+d) + \pi i S(a,b,c,d),
\end{equation}
where the real numbers $S(a,b,c,d)$  are independent of $z$.  Indeed, Dedekind obtained the evaluation that
\begin{equation}\label{Dsum}
S(a,b,c,d)= \frac{a+d}{12c}-\frac{c}{|c|}\left( \frac{1}{4} + s(d,|c|)\right) \qquad \text{for} \qquad c\neq 0
\end{equation}
 and
 \begin{equation}\label{Dsum2}
S(a,b,0,d)= \frac{b}{12d}+\frac{d/|d|-1}{4}.
\end{equation}
The term $s(d,|c|)$ in \eqref{Dsum} is a {\em Dedekind sum} which is defined as
\begin{equation}\label{dede}
    s(d,c):= \sum_{m=0}^{c-1} \biggl(\biggl( \frac{dm}{c}\biggr)\biggr) \biggl(\biggl(  \frac{m}{c}\biggr)\biggr) \quad \text{for} \quad
    \left(\left( x\right)\right) := \begin{cases} x - \lfloor x\rfloor - 1/2 & \text{if} \quad x \in \R, x\not\in \Z\\
    0 & \text{if} \quad x \in \Z
    \end{cases}
\end{equation}
for $d,c$ coprime and  $c\geq 1$; see \cite[Chap. IX]{La95}. It is immediate that the periodicity relation
\begin{equation}\label{perio}
s(d+c,c)=s(d,c)
\end{equation}
holds.  Dedekind showed that (\ref{log_eta}) implies the reciprocity law
\begin{equation}\label{recip}
s(c,d)+s(d,c)=\frac{1}{12}\left(\frac{c}{d}+\frac{d}{c} + \frac{1}{cd}\right)-\frac{1}{4}
\end{equation}
for $c$ and $d$  relatively prime positive integers.
Several other proofs of the identity (\ref{recip}) are given in \cite{RG72}.

\subsection{Generalizations of Dedekind sums}

There are various directions in which  authors have obtained generalizations of Dedekind sums, including
sums of Bernoulli polynomials, cotangent sums, and rational functions evaluated at roots of unity.  Many of these
investigations take the evaluation (\ref{dede}) as a starting point.  Alternatively, one can begin with (\ref{log_eta}),
meaning that one defines $S$ as coming from the root of unity that appears when considering the transformation
property of the logarithm of a modular form, or, more generally, the holomorphic section of a power of the canonical bundle
twisted by a flat line bundle.

The Dedekind eta function (\ref{etad})  arises naturally in the constant term of
a Laurent expansion for the non-holomorphic Eisenstein series for $\SL(2, \Z)$. In other words, $\eta(z)$ appears in terms of a Kronecker limit function;
see, for example, \cite{DIT} for a discussion of several intricate connections between Kronecker limit functions and number theory.
Goldstein in \cite{Gn1,Gn2,Gn3} takes this point of view and studies Kronecker limit functions of Eisenstein series associated to general Fuchsian groups
$\G$.  In particular, Goldstein develops the notion of generalized Dedekind sums associated to $\G$.   A deeper study in this direction is given in
\cite{Bu} and \cite{Bu2} where the author also studies distribution results of general Dedekind sums, which were previously only known in a few instances.

We will use the notation $S_\G$ for the generalization of $S$ from \e{log_eta} to other groups $\G$, and call this object a {\em modular Dedekind symbol}.
We will use $H_\G$ for the corresponding {\em generalized Dedekind sum} extending \e{dede}. The periodicity \e{perio} may be traced back to the parabolic
generator $(\smallmatrix 1 & 1 \\ 0 & 1 \endsmallmatrix )$ of $\SL(2, \Z)$. Similarly, the reciprocity \e{recip} comes from the elliptic generator
$(\smallmatrix 0 & -1 \\ 1 & 0 \endsmallmatrix )$. The subgroups $\G$ of $\SL(2, \R)$ which we consider will be finitely generated by parabolic, elliptic and
hyperbolic elements. For each parabolic generator we will obtain a {\em periodicity relation} which generalizes \e{perio}, while each elliptic generator
will provide a {\em reciprocity relation} as in \e{recip}. There will also be {\em hyperbolic periodicity relations} from the hyperbolic generators,
and our focus of interest is to describe all of these transformations.

The above ideas also extend naturally to a type of non-holomorphic Eisenstein series that includes modular symbols in its definition;
see \e{mod} for the definition of the modular symbol pairing.  The simplest such series, which contains one modular symbol, was studied in \cite{OS00}, for example,
and its Kronecker limit was obtained in \cite{JO'S05}. The authors in \cite{JO'SS1} defined and studied a new modular Dedekind symbol, denoted by $S^*_\G$, obtained from the Kronecker limit of Eisenstein series constructed with the square of the absolute value of a modular symbol.
These types of series are not automorphic in the usual sense but rather transform with a higher-order condition as described in \e{higher}. Recently in \cite{PR}
these higher-order Eisenstein series played a key role in the proof of conjectures by Mazur,
Rubin and Stein on the statistics of modular symbols. These series  have been generalized further in \cite{CHO}.

\subsection{Definitions and main results} \label{dfg}
 Much of the following background material is found in the early chapters of \cite{Iw2, S}, for example, but
 is included here for the convenience of the reader.

 The group
$\SL(2,\R)$ acts by linear fractional transformations on
$\H$, and this action may be extended to $\H \cup \R \cup\{\ci\}$.
Let $\Gamma$ be a discrete subgroup of $\SL(2,\R)$ where the quotient $\G\backslash \H$ has   finite hyperbolic volume $V_\G$.
This means that $\G$ is a  Fuchsian group of the first kind. We assume that $\G$ contains at least one parabolic element fixing a cusp in
$\R \cup\{\ci\}$. If $\ca$ is a cusp of $\G$ then we label the subgroup of parabolic elements fixing  $\ca$ as
$\G_\ca$. The image of $\G_\ca$
 under the map $\SL(2,\R) \to \SL(2,\R)/\pm I$ is isomorphic to $\Z$. 
Denoting this image by $\overline{\G}_\ca$, there exists
a scaling matrix $\sa \in \SL(2,\R)$ such that
\begin{equation}\label{scal}
  \sa \infty = \ca \quad \text{and} \quad \sai\overline{\G}_\ca \sa  = \left\{\left. \pm \left(\smallmatrix 1 & \ell \\ 0 & 1 \endsmallmatrix \right) \
   \right|\  \ell\in \Z\right\}.
\end{equation}
The matrix $\sa$ is unique up to multiplication on the right by $\pm\left(\smallmatrix 1
& t \\ 0 & 1 \endsmallmatrix\right)$ for any  $t\in \R$.
By conjugating $\G$, if necessary, we may assume
that $\G$ has a cusp at infinity,  which is denoted  by $\infty$, with
scaling matrix $\sigma_{\infty}$ equal to the identity matrix $I$. Therefore $\overline{\G}_\ci$ is generated by
$\pm P_{\infty}$ for $P_{\infty} := \left(\smallmatrix 1 & 1 \\ 0 & 1 \endsmallmatrix \right)$.
We also assume for simplicity that
$P_{\infty}$ is in $\G$. It follows that if $-I \in \G$ then $\G_\ci$ is generated by $P_\ci$ and $-I$.   If
$-I \not\in \G$ then $\G_\ci$ is just generated by  $P_\ci$. Note that we are excluding the case where $-I$ and $P_\ci$ are both not
in $\G$ and  $-P_\ci$ generates $\G_\ci$. On the other hand, if $\gamma\in\Gamma$ has bottom left entry equal to $0$, then $\gamma\in\Gamma_\infty$. This is clear for $\G = \SL(2,\Z)$ and true more generally for any discrete subgroup of $\SL(2,\R)$; see Proposition 1.17 of \cite{S}.

The modular Dedekind symbols $S_\G$ and $S^{*}_\G$, associated to $\G$, are defined in
\cite[Sect. 1.3]{JO'SS1}. The definitions and relevant details are reviewed in Section 2.  For now, recall that $S^*_\G$ requires modular
symbols for its construction, and these  depend on the choice of a nonzero weight $2$ holomorphic cusp form for $\G$.
The space of such forms, $S_2(\G)$, has dimension equal to the genus $g$ of $\GH$. Therefore, $S^*_\G = S^*_{\G,f}$
depends on a group  $\Gamma$ of positive genus and a nonzero $f \in S_2(\G)$.

Let $\theta_\G:=S^*_\G-S_\G$ and set $ A_\G := V_{\Gamma}/4\pi$. It follows from the Gauss--Bonnet formula \eqref{gbon}
that $A_\G$ is rational. For $r\in \R$, define $\sgn(r) := r/|r|$ if $r\neq 0$ and $\sgn(0):=0$. We fix $\G$ and will usually
omit the use of the symbol $\G$ from the notation in what follows.

\begin{defs}[Generalized Dedekind sums]
Let $\gamma =\left(\smallmatrix a & b \\ c & d  \endsmallmatrix \right) \in \Gamma$. If $c \neq 0$ then define
$$ H(\gamma) :=  - S(\gamma)+ A \frac{a +d }{c }  -\frac{1}4 \sgn(c) \quad \text{and}\quad  H^{*}(\gamma) :=  \theta(\gamma) +  A \frac{a +d }{c } .$$
If $c=0$, set
$$
H(\gamma) :=  - S(\gamma)+ A \frac{b}{d } +\frac{1}4(\sgn(d)-1) \quad \text{and}\quad  H^{*}(\gamma) :=  \theta(\gamma) +  A \frac{b}{d } .
$$
\end{defs}

When $\Gamma$ is the modular group $\mathrm{SL}(2,\Z)$, the constant $A$ is $1/12$.  In this case, the above definition coincides with
\eqref{Dsum} and one gets that
$H(\g)= \frac{c}{|c|} s(d,|c|)$ for $c\neq 0$ and $H(\g)= 0$ for $c = 0$.  Therefore,  $H(\g)$ may be viewed as a \textit{signed} generalization of the Dedekind sum,
as introduced in \cite{Bu}.

\begin{prop}\label{period}
The functions $H$ and $H^*$ are left $\G_\ci$-invariant. Therefore, $H$ and $H^*$ depend only on the bottom rows of their arguments,
thus allowing us to write
\begin{equation}\label{notnx}
  H(d,c):=H(\gamma)
\quad \textrm{and}
\quad
H^*(d,c):=H^*(\gamma)
\end{equation}
for $\gamma=(\begin{smallmatrix} *&*\\c&d\end{smallmatrix}) \in \G$.
In this form one has the periodicity relations
\begin{equation}\label{periodx}
  H(d+c,c)=H(d,c)
\quad \textrm{and}
\quad
H^*(d+c,c)=H^*(d,c).
\end{equation}
\end{prop}


The values taken by $H$ for certain groups $\G$ have been studied by many authors. On the other hand, the values of $H^*$ are just beginning to be understood.
 It follows from Theorem 4 of \cite{JO'SS1}
that, in the case of genus one Hecke congruence groups, $H^*$ is always rational. These are exactly the groups $\Gamma_0(N)$ for which
$N\in\{11,14,15,17,19,20,21,24,27,32,36,49\}$. See Section \ref{gam11} for a detailed description of $H^*$ when $\G=\G_0(11)$, for example. 

The next theorem gives a density result for the genus one Hecke congruence groups mentioned above.
For $r\in \R$, let $\{r\}$ denote its fractional part.

\begin{theorem}
Let $\Gamma$ be a genus one Hecke congruence group. Then there exists $t=t(\Gamma)\in\N$ with the following property.
For all $ d,c $ such that $ 0<d\leq c\leq X$ and $
(\begin{smallmatrix} *&*\\c & d \end{smallmatrix}) \in \G$,  the points
$$
\left( \{ d/c\}, \{t \cdot H^*(d,c)\} \right)
$$
become equidistributed in the square $\R/\Z\times\R/\Z$ as $X\to\infty$.
\end{theorem}

For $\Gamma=\SL(2,\Z)$, Myerson \cite{My} proved a more general version of this result for $H$, with $t\in\N$ replaced by any nonzero $t\in\R$.
Also, for any cofinite Fuchsian group $\Gamma$ and any nonzero $t\in\R$, it is shown in \cite{Bu} that $\{t H(d,c)\}$ becomes equidistributed in $\R/\Z$.

We now turn our attention to the transformation laws of the generalized Dedekind sums $H$ and $H^*$. In what follows, we  use the notation $
\left(\begin{matrix} a_{\gamma} & b_{\gamma}  \\ c_{\gamma} & d_{\gamma}  \end{matrix} \right)
$ for the entries of matrix $\g$.

\begin{theorem}\label{tr-thm}
Let $\g$, $\tau \in \Gamma$ with  $c_{\g}c_{\tau}c_{\g\tau}\neq 0$. Then
\begin{align}\label{hhx1}
  H(\g\tau)-H(\g)-H(\tau) & = -A\left(\frac{c_\gamma}{c_\tau c_{\gamma\tau}} + \frac{c_\tau}{c_\gamma c_{\gamma\tau}} +\frac{c_{\gamma\tau}}{c_\gamma c_\tau}\right) + \frac 14 \sgn(c_\gamma c_\tau c_{\gamma\tau}),  \\
  H^*(\g\tau)-H^*(\g)-H^*(\tau) & = -A\left(\frac{c_\gamma}{c_\tau c_{\gamma\tau}} + \frac{c_\tau}{c_\gamma c_{\gamma\tau}} +\frac{c_{\gamma\tau}}{c_\gamma c_\tau}\right) +
  \frac{V_f}{2\pi}\mathrm{Im}\left(\langle\gamma,f\rangle\overline{\langle \tau,f\rangle}\right) \label{hhx2}
\end{align}
where $V_f := V_\G/(16 \pi^2 \| f\|^2)$, $\|f\|$ is the Petersson norm of $f$, and $\langle \cdot , f\rangle$ is the modular symbol pairing as defined in (\ref{mod}).
\end{theorem}

The identity \eqref{hhx1} first appeared in \cite[Sect. 2.4]{Bu2}. By writing $\g_1=\g$, $\g_2=\tau$ and $\g_3=(\g\tau)^{-1}$, the above transformation laws may be given the following elegant formulation, first observed by Dieter in \cite{Di} for the classical Dedekind sums (\ref{dede}).

\begin{cor}\label{3term}
Let $\gamma_1=(\begin{smallmatrix} *&*\\c_1&*\end{smallmatrix}),\gamma_2=(\begin{smallmatrix} *&*\\c_2&*\end{smallmatrix}),\gamma_3=(\begin{smallmatrix} *&*\\c_3&*\end{smallmatrix})\in\Gamma$ be such that $c_1 c_2 c_3\neq0$ and $\gamma_1\gamma_2\gamma_3=I$. Then
\begin{align}\label{hjk}
  H(\gamma_1)+H(\gamma_2)+H(\gamma_3) & = -A\left(\frac{c_1}{c_2 c_3}+\frac{c_2}{c_1 c_3}+\frac{c_3}{c_1 c_2}\right) + \frac{1}{4} \sgn(c_1 c_2 c_3), \\
  H^*(\gamma_1)+H^*(\gamma_2)+H^*(\gamma_3) & = -A\left(\frac{c_1}{c_2 c_3}+\frac{c_2}{c_1 c_3}+\frac{c_3}{c_1 c_2}\right) -\frac{V_f}{2\pi}\mathrm{Im}\left(\langle\gamma_1,f\rangle\overline{\langle \gamma_2,f\rangle}\right). \label{di2}
\end{align}
\end{cor}

Equation \eqref{hjk} is Theorem 2 of \cite{Bu2}, with the signs on the right corrected.
These three-term relations also imply succinct two-term relations, though it is simpler to prove the following directly.

\begin{theorem} \label{2trm}
Let $\g$, $\tau \in \Gamma$ with  $c_{\g}c_{\tau}c_{\g\tau}\neq 0$. Then, using the notation of \e{notnx},
\begin{align}\label{recip2a}
  H(d_{\g},c_{\g}) - H(d_{\g\tau},c_{\g\tau})  &  =
A\left(\frac{d_{\g}}{c_{\g}} - \frac{d_{\g\tau}}{c_{\g\tau}} +
\frac{c_{\tau}}{c_{\g}c_{\g\tau}}\right) + S(\tau) +\frac{1}4\big(\sgn(c_\tau) - \sgn(c_\gamma c_\tau c_{\gamma\tau})\big), \\
   H^*(d_{\g},c_{\g}) - H^*(d_{\g\tau},c_{\g\tau})  &  =
A\left(\frac{d_{\g}}{c_{\g}} - \frac{d_{\g\tau}}{c_{\g\tau}} +
\frac{c_{\tau}}{c_{\g}c_{\g\tau}}\right) -\theta(\tau)-\frac{V_f}{2\pi}\mathrm{Im}\left(\langle\gamma,f\rangle\overline{\langle \tau,f\rangle}\right). \label{recip3}
\end{align}
\end{theorem}

If $\tau$ is elliptic or parabolic   then $S(\tau)$ in \eqref{recip2a} is not difficult to evaluate,
from which one obtains reciprocity and periodicity  relations for $H$. If $\tau$ in \eqref{recip3} is elliptic or parabolic,
and not fixing a cusp equivalent to $\infty \bmod \G$,  then we will see that only the first term
on the right can be nonzero, thus giving the simple law
\begin{equation}\label{recip2b}
  H^*(d_{\g},c_{\g}) - H^*(d_{\g\tau},c_{\g\tau})    =
A\left(\frac{d_{\g}}{c_{\g}} - \frac{d_{\g\tau}}{c_{\g\tau}} +
\frac{c_{\tau}}{c_{\g}c_{\g\tau}}\right).
\end{equation}


Among the various relations  obtained, we highlight here an interesting identity, involving the
classical Dedekind sum \eqref{dede} and the M\"obius function $\mu$, which we derive from generalized Dedekind sums
associated to the group $\G_{0}(N)$ for any integer $N\geq 2$.  Let $c$, $d$ and $k$ be positive
integers with $c$ and $d$ relatively prime and $N | c$. Then, as seen in \e{rty},
$$
\sum_{v\mid N} \frac{\mu(v)}{v}\left( s\left( d, \frac{v c}{N}\right) - s\left( d, \frac{v c}{N} + v k d\right) \right) =
\frac{N^2 k(d^2+1)}{12c(c+kNd)}  \prod_{p\mid N} \left(1-p^{-2} \right).
$$

The following material is divided into four sections.  Section 2  presents initial results needed for our analysis.  The proofs of the main results are given in Section 3.
Section 4 provides additional transformation formulas, and in Section 5 we give some examples of the theory
for specific groups.

Finally, we note that our definitions of $S$, $S^*$, $H$ and $H^*$ above are for the cusp at infinity. There
are analogous definitions for  arbitrary cusps and our formulas may be extended to cover these cases. For  details on the general parametrization of the
modular Dedekind symbols $S$ and $S^*$, see \cite[Sect. 5]{JO'SS1}. For the general parametrization of $H$, see \cite{Bu}.

\section{Preliminaries}

In this section, for the sake of completeness, we  recall results from 
\cite{Bu2}, \cite{Gn1} and \cite{JO'SS1} that will be required. Throughout, as described in Section \ref{dfg}, $\G$ is a discrete subgroup of $\SL(2,\R)$ with finite volume and at least one cusp, at $\ci$. It is assumed that the parabolic subgroup fixing $\ci$ is generated by $P_{\infty} := \left(\smallmatrix 1 & 1 \\ 0 & 1 \endsmallmatrix \right)$, and also $-I$ if $-I\in \G$.

\subsection{The phase factor}

For  any two matrices $M$, $N$ in $\SL (2,\R)$ define the {\em phase factor} $\omega(M,N)$ by
 \begin{equation}\label{w1}
\omega(M,N):=\bigl(-\log j(MN, z) + \log j(M, N z)+ \log j(N, z)\bigr)/(2\pi i).
\end{equation}
We are using the principal branch of the logarithm with argument in $(-\pi,\pi]$. The right side of (\ref{w1}) is independent of $z \in \H$ and may only take  values in the set $\{-1,0,1\}$.
This phase factor  is treated carefully by Petersson in \cite{Pet} where an explicit formula for evaluation of $\omega(M,N)$ is given
in terms of the bottom rows of $M,$ $N$ and $MN$.
Here, we will recall a slightly different formulation  from Proposition 3.2 of \cite{JO'SS1}  which will be used throughout this paper.

For all $M, N$ in $\mathrm{SL}(2,\R)$ set $\mathbf{c} =(\sgn ( c_M), \sgn ( c_N), \sgn ( c_{MN}))$. Then
\begin{equation}\label{omg}
\omega(M,N)=\begin{cases}
1 & \textrm{ if \ } \mathbf{c} = (1,1,-1)  \text{ or } (0,1,-1) \text{ or } (1,0,-1)\\
1 & \textrm{ if \ } \mathbf{c} = (0,0,0) \textrm{ and } \sgn( d_M)=\sgn( d_N )=-1 \\
-1 & \textrm{ if \ } \mathbf{c} = (-1,-1,1) \text{ or } (-1,-1,0) \\
0 & \textrm{ otherwise. }
\end{cases}
\end{equation}
Applying the above relations, it is straightforward to see that for $M\in \SL(2,\R)$ and $k \in \Z$
one has
\begin{align}\label{phs}
  \omega(P_\ci^k,M)=\omega(M,P_\ci^k) &=0, \\
  \omega(-P_\ci^k,M)=\omega(M,-P_\ci^k) &= \left\{
               \begin{array}{ll}
                 (1+\sgn(c_M))/2, & \text{   if $c_M\neq 0$}; \\
                 (1-\sgn(d_M))/2, & \text{  if $c_M = 0$.}
               \end{array}
             \right. \label{phs2}
\end{align}
The next formula of Petersson may also be verified easily. For $c_M c_N c_{MN} \neq 0$,
\begin{equation}\label{phs3}
  \omega(M,N)=\left( \sgn(c_M) + \sgn(c_N) -\sgn(c_{M N})-\sgn(c_M c_N c_{M N}) \right)/4.
\end{equation}

\subsection{The modular Dedekind symbol $S$}

The non-holomorphic Eisenstein series associated to $\G$ and the cusp at $\infty$ is defined for $z\in \H$ and $s\in \C$ with $\textrm{Re}(s) > 1$ by
$$
E_\infty(z,s) :=\sum_{\g \in \Gamma_\infty\backslash\G} \Im(\g z)^s.
$$
This  series admits a meromorphic continuation to all $s \in \C$; at $s=1$, the Laurent expansion has a first order pole, and the
Kronecker limit function is defined to be the next order term in the  expansion.  The Kronecker limit function can be shown to
have the form $\log \vert \textrm{Im}(z) \eta_{\Gamma,\infty}(z) \vert$ for some function $\eta_{\Gamma,\infty}(z)$ which is
non-vanishing on $\H$, locally holomorphic, and transforms as a weight $1/2$ modular form with respect to $\G$.  This modular
transformation includes the multiplier system $e^{\pi i S_{\Gamma, \ci}(\g)}$.  In other words, we have that
\begin{equation} \label{etatran}
\log\eta_{\Gamma,\infty}(\g z)=  \log\eta_{\Gamma,\infty}(z) + \frac{1}{2} \log j(\g, z) + \pi i S_{\Gamma,\infty}(\g)\quad \text{for all} \quad \g \in \G
\end{equation}
where $j(\g,z):=c_\g z+d_\g$.
In the case when $\Gamma=\mathrm{SL}(2,\Z)$, the function $\eta_{\Gamma,\infty}(z)$ is the classical Dedekind eta function \eqref{etad}, and \eqref{etatran} coincides with \eqref{log_eta}. In the sequel, we consider the group $\Gamma$ and the cusp at $\infty$ as fixed, so we will write $S(\g)=S_{\Gamma,\infty}(\g)$ and call this function the \textit{modular Dedekind symbol} for $\G$. This construction was initiated in \cite{Gn1} and \cite{Gn2}.

Define an automorphism $\iota$ of $\SL(2,\R)$ as
\begin{equation*} \label{iota def}
\left(\begin{matrix} a&b\\c&d\end{matrix}
                            \right) \stackrel{\iota}{\longrightarrow}
\left(\begin{matrix} a&-b\\-c&d\end{matrix}
                            \right).
\end{equation*}
Also set
\begin{align}\label{rho}
\rho(\gamma) :=\left\{
               \begin{array}{ll}
                 1, & \text{   if } c_\g=0 \text{   and   } d_\g<0; \\
                 0, & \text{  otherwise.}
               \end{array}
             \right.
\end{align}
The properties we will need for the symbol $S$ are assembled in the next theorem.

\begin{theorem} \label{sa}
For all $\gamma$ and $\tau$ in $\Gamma$ the following assertions are true.
\begin{itemize}
\item [i)] We have that $S(\g \tau)=S(\g)+S(\tau) + \omega(\g , \tau).$
\item [ii)] We have $S(I)=0$, and if $-I \in\Gamma$, then $S(-I)=-1/2$.
\item [iii)] If $\iota(\Gamma)=\Gamma$, then $S(\iota(\g))=-S(\g)-\rho(\g)$, where
$\rho$ is defined in \eqref{rho}.
\item[iv)] Suppose $\g$ is  parabolic, fixing the cusp $\ca$. Choose a scaling matrix $\sigma_\ca$ satisfying \eqref{scal}. Then    $\g_\ca:=\sigma_\ca ^{-1}\g \sigma_\ca= \pm(\smallmatrix 1
& h \\ 0 & 1 \endsmallmatrix )$ for a well-defined $h\in \Z$ and
\begin{equation}\label{comp}
  S(\g)=  \delta(\infty,\ca) \cdot h A   -\log(j(\g_\ca,i))/2\pi i-  \omega(\sai,\g)+ \omega(\g_\ca,\sai).
\end{equation}
 Here $\delta(\infty,\ca)=1$ in the case $\ca\equiv \infty \bmod \G$ and $\delta(\infty,\ca)=0$ otherwise.
\item [v)] Let $\g$ be  elliptic, fixing $z_0 \in \H$ and with $\g^r=I$ for $r>0$. Then $$S(\g) = -\frac 1{2\pi i} \log j( \g, z_0) =  -\frac{1}{r} \sum_{k=1}^{r-1} \omega(\g^k, \g ).$$
\end{itemize}
\end{theorem}
\begin{proof}
All these results are proved in \cite{JO'SS1}.
Part i) is shown  there in Theorem 1. The first assertion of part ii) follows immediately from part i) and the second assertion  is  \cite[Eq. (3.9)]{JO'SS1}. Propositions 6.3, 5.6 and 5.8 of  \cite{JO'SS1} give parts iii), iv) and v) respectively.
\end{proof}
Note that equation \eqref{comp} above simplifies to
\begin{equation}\label{comp2}
  S(\g)=  \delta(\infty,\ca) \cdot h A
\end{equation}
when $\g_\ca=(\smallmatrix 1
& h \\ 0 & 1 \endsmallmatrix )$; see Proposition 5.6 of \cite{JO'SS1}.

\subsection{The higher-order modular Dedekind symbol $S^*$}

Assume now that the group $\Gamma$ has positive genus. This implies that the space $S_2(\Gamma)$ of weight two holomorphic cusp forms with respect to $\Gamma$ has positive dimension.  Let $
f\in S_2(\Gamma)$ be a nonzero cusp form which remains fixed throughout the discussion.

We define a non-holomorphic Eisenstein series at the cusp at infinity twisted by powers of modular symbols  with
\begin{equation}\label{E m,n}
E^{m,n}_\infty(z,s;f):= \sum_{\g \in \G_\ci\backslash\G} \langle\g,f\rangle^{m}\overline{\langle\g,f\rangle}^{n} \Im(\g z)^s,
\end{equation}
for $\mathrm{Re}(s)>1$. The {\em modular symbol  pairing} is defined by
\begin{equation}\label{mod}
    \s{\g}{f} := 2\pi i\int_{z_0}^{\g z_0} f(w) \, dw.
\end{equation}
Since $f \in S_2(\G)$, the modular symbol pairing is
independent of $z_0 \in \H \cup \{\mathtt{the\,\, cusps\,\, of}\,\, \G\}$, and the pairing
satisfies the identity
\begin{equation}\label{modu}
  \s{\g\tau}{f}=\s{\g}{f}+\s{\tau}{f} \quad \text{for all} \quad \g, \tau \in \G.
\end{equation}
If $\g$ is parabolic or elliptic, then letting $z_0$ be the point fixed by $\g$ shows that $\s{\g}{f}=0$. In particular, \e{E m,n} is well-defined.

Let $\g \in \G$ act on a map $\psi$ by $(\psi|\g)(z):=\psi(\g z)$ and extend this action to all $\C[\G]$ by linearity. We may call  $\psi$ an {\em $n$th-order automorphic form} if
\begin{equation} \label{higher}
    \psi \big|(\g_1 -I)(\g_2-I) \cdots (\g_n-I) = 0  \quad \text{for all} \quad \g_1,\g_2, \dots,\g_n \in \G.
\end{equation}
With this notation, $\G$-invariant functions such as $E_\ci(z,s)$ are first-order. The series $E^{m,n}_\infty(z,s;f)$ has order $m+n+1$. See for example \cite[Sect. 5.1]{CHO} for more details.

For $m=n$
the series \eqref{E m,n} has a pole at $s=1$ of order $m+1$.  An examination of the Kronecker limit formula for $E^{m,m}_\infty(z,s;f)$  as $s\to 1$,
as developed in \cite[Sect. 4]{JO'SS1}, proves the following: there exists a function $\eta^*_{\Gamma,\infty,f}$ which, up to a factor of modulus one, is a modular form of weight $1/2$ with respect
to $\Gamma$, and satisfies the transformation formula
\begin{equation} \label{etatranstar}
\log\eta^*_{\Gamma,\infty,f}(\g z)=  \log\eta^*_{\Gamma,\infty,f}(z) + \frac{1}{2} \log j(\g, z) + \pi i S^*_{\Gamma,\infty,f}(\g)\quad \text{for all} \quad \g \in \G.
\end{equation}
The terms $S^*_{\Gamma,\infty,f}(\g)$ turn out to be independent of $m$ for all $m\geq 1$, and so
the function $S^*=S^*_{\Gamma,\infty,f}$ is called a {\em higher-order modular Dedekind symbol} in \cite{JO'SS1}.
It  possesses the following properties.

\begin{theorem} \label{sa2}
For all $\gamma$ and $\tau$ in $\Gamma$ the following assertions are true.
\begin{itemize}
\item [i)] We have that $S^*(\g \tau)=S^*(\g)+S^*(\tau) + \frac{V_f}{2\pi} \Im \big(\langle\g,f\rangle \overline{\langle\tau,f\rangle}\big) + \omega(\g ,\tau)$ where $V_f := V_\G/(16 \pi^2 \| f\|^2)$ and $\|f\|$ denotes the Petersson norm of $f$.
\item [ii)] If $-I \in\Gamma$ then $S^*(-I)=-1/2$.
\item [iii)] If $\iota(\Gamma)=\Gamma$, then $S^*(\iota(\g))=-S^*(\g)-\rho(\g)$.
\item[iv)] Suppose $\g$ is  parabolic, fixing the cusp $\ca$. Choose a scaling matrix $\sigma_\ca$ satisfying \eqref{scal} and set $\g_\ca:=\sigma_\ca ^{-1}\g \sigma_\ca$. Then
\begin{equation}\label{comp2x}
  S^*(\g)=     -\log(j(\g_\ca,i))/2\pi i-  \omega(\sai,\g)+ \omega(\g_\ca,\sai).
\end{equation}
\item [v)] If $\g$ is  elliptic then $S^*(\g) = S(\g)$.
\end{itemize}
\end{theorem}
\begin{proof}
Part i) is given in Theorem 2 of \cite{JO'SS1}. Part ii) is contained in Corollary 5.4 there. Propositions 6.2, 5.7 and 5.8 of \cite{JO'SS1} show parts iii), iv) and v) respectively.
\end{proof}

\subsection{The difference $\theta:=S^*-S$}

We now consider the difference $\theta:=S^*-S$ from which the generalized Dedekind sum $H^*$ is constructed. We could have used just $S^*$ to define $H^*$, but the advantage of using $\theta$ is that all the phase factors disappear, as we see now.

\begin{cor} \label{thec}
For all $\gamma$ and $\tau$ in $\Gamma$ the following assertions are true.
\begin{itemize}
\item[i)] We have that $\theta(\gamma\tau)=\theta(\gamma)+\theta(\tau)+\tfrac{V_f}{2\pi} \mathrm{Im}
    \big(\langle \gamma,f\rangle\overline{\langle \tau,f\rangle}\big)$.

\item[ii)] If $-I\in \Gamma$ then $\theta(-I)=0$, and  $\theta(-\gamma)=\theta(\gamma)$.

\item[iii)] If $\iota(\Gamma)=\Gamma$, then  $\theta(\iota(\gamma))=-\theta(\gamma)$.

\item[iv)] Assume $\gamma$ is a parabolic element fixing a cusp $\mathfrak{a}$. Choose a scaling matrix $\sigma_\ca$ satisfying \eqref{scal}. Then    $\sigma_\ca ^{-1}\g \sigma_\ca= \pm(\smallmatrix 1
& h \\ 0 & 1 \endsmallmatrix )$ for a well-defined $h\in \Z$. We have $\theta(\gamma) =-\delta(\infty,\ca)  \cdot hA$ where $\delta(\infty,\ca)=1$ in the case $\ca\equiv \infty \bmod \G$ and $\delta(\infty,\ca)=0$ otherwise.

\item[v)] If $\g$ is an elliptic element, then $\theta(\g) = 0.$

\item[vi)] If $\gamma$ or $\tau$ is either elliptic or parabolic, then $\theta(\gamma\tau)=\theta(\gamma)+\theta(\tau)$ since
modular symbols are zero for elliptic or parabolic elements.
\item[vii)] For each $n\in\Z$, $\theta(\gamma^n)=n\cdot \theta(\gamma)$.
\end{itemize}
\end{cor}

Before continuing, we describe some further properties of this map.
The following result states that, in effect, $\theta$ detects the difference between $\Gamma$ and
its abelianization.

\begin{lemma}\label{prop-theta}
For all $\gamma$ and $\tau$ in $\Gamma$,
$$
\theta(\gamma\tau)-\theta(\gamma)-\theta(\tau) = \tfrac12  \theta([\gamma,\tau]),
$$
where the commutator $[\gamma,\tau]$ is defined as $[\gamma,\tau] :=\gamma\tau\gamma^{-1}\tau^{-1}$. 
\end{lemma}

\begin{proof}
We first claim that
\begin{equation}\label{skp}
  \theta([\gamma,\tau]) = \tfrac{V_f}{\pi}\mathrm{Im}\big(\langle\gamma,f\rangle\overline{\langle\tau,f\rangle}\big).
\end{equation}
For this
\begin{align*}
\theta([\gamma,\tau]) &= \theta(\gamma\tau\gamma^{-1})-\theta(\tau) = \theta(\gamma\tau) -\theta(\gamma)-\theta(\tau)-\tfrac{V_f}{2\pi}\mathrm{Im}\big(\langle\tau,f\rangle\overline{\langle\gamma,f\rangle}\big)\\
&= \tfrac{V_f}{2\pi}\mathrm{Im}(\langle\gamma,f\rangle\overline{\langle\tau,f\rangle})-\tfrac{V_f}{2\pi}\mathrm{Im}
\big(\langle\tau,f\rangle\overline{\langle\gamma,f\rangle}\big)\\
&= \tfrac{V_f}{\pi}\mathrm{Im}\big(\langle\gamma,f\rangle\overline{\langle\tau,f\rangle}\big).
\end{align*}
Using \eqref{skp} in Corollary \ref{thec} part i) completes the proof.
\end{proof}

The following corollary justifies the use of the term \textit{sum} when referring to $H^*$.

\begin{cor}\label{iteration}
For any word $\gamma_1^{k_1}\cdots \gamma_m^{k_m}$ in $\G$
$$
\theta(\gamma_1^{k_1}\cdots \gamma_m^{k_m}) =  k_1\cdot\theta(\gamma_1)+\dots +k_m\cdot \theta(\gamma_m) + \tfrac12 \sum_{1\leq i<j\leq m} k_i k_j\cdot \theta([\gamma_i,\gamma_j]).
$$
\end{cor}

\begin{proof}
The assertion follows by iterating  Lemma \ref{prop-theta} and simplifying with \eqref{skp} and \e{modu}.
\end{proof}

We note here that $\theta$ can be viewed as the \textit{homogenization} of the 
generalized Dedekind sum $H^*$.

\begin{prop}
For all $\tau \in \G$ with $\tau \not\in \G_\ci$ we have
\begin{equation}\label{lmh}
  \lim_{n\to\infty}\frac{H^*(\tau^n)}{n} = \theta(\tau).
\end{equation}
\end{prop}

\begin{proof}
We will use the notation $(\begin{smallmatrix} a_n & b_n \\c_n& d_n \end{smallmatrix})$ for the entries of $\tau^n$. If $\tau=\pm I$ then \eqref{lmh} is true since $H^*(\tau^n)=\theta(\tau)=0$. Now assume that $\tau$ is parabolic, so that $|\tr(\tau)|=2$. We may choose $\sigma \in \SL(2,\R)$ such that $\sigma^{-1}\tau \sigma = \pm (\begin{smallmatrix} 1&h\\0&1\end{smallmatrix})$ for nonzero $h\in \Z$. If we write $\sigma=(\begin{smallmatrix} \alpha & \beta\\ \g & \delta\end{smallmatrix})$ then
$$
\tau^n = \pm \sigma \left(\begin{matrix} 1 & n h \\ 0 & 1 \end{matrix} \right) \sigma^{-1}
= \pm \left(\begin{matrix} 1-n h \alpha \g & n h \alpha^2 \\ -n h \g^2 & 1+nh \alpha \g \end{matrix} \right).
$$
Since $\tau \not\in \G_\ci$ we have  $c_1 \neq 0$ and therefore $c_n \neq 0$ for all $n$. Hence
$$
H^*(\tau^n) = \theta(\tau^n)+A \frac{a_n+d_n}{c_n} = n \theta(\tau)+A \frac{2}{-nh \gamma^2}
$$
and \eqref{lmh} follows.

In the remaining case we have $|\tr(\tau)|\neq 2$ and so $c_1\neq0$ (cf.~Proposition 1.17 of \cite{S}), $\tau$ has two distinct eigenvalues $\lambda_1$, $\lambda_2$ and is diagonalizable. We find
\begin{equation}\label{pli}
  \tau^n = \pm \sigma \left(\begin{matrix} \lambda_1^n & 0 \\ 0 & \lambda_2^n \end{matrix} \right) \sigma^{-1}
= \pm \left(\begin{matrix} \lambda_2^n+ \alpha \delta(\lambda_1^n-\lambda_2^n) &-\alpha \beta(\lambda_1^n-\lambda_2^n) \\  \g \delta(\lambda_1^n-\lambda_2^n) & \lambda_1^n- \alpha \delta(\lambda_1^n-\lambda_2^n) \end{matrix} \right).
\end{equation}
This implies the formulas (taking the positive sign in \eqref{pli}, otherwise multiply by $-1$)
\begin{alignat*}{2}
a_n & =\lambda_2^n +(a_1-\lambda_2)\frac{\lambda_1^n-\lambda_2^n}{\lambda_1-\lambda_2}, & \qquad b_n & =b_1 \frac{\lambda_1^n-\lambda_2^n}{\lambda_1-\lambda_2},\\
c_n & =c_1 \frac{\lambda_1^n-\lambda_2^n}{\lambda_1-\lambda_2}, & d_n & =\lambda_1^n +(d_1-\lambda_1)\frac{\lambda_1^n-\lambda_2^n}{\lambda_1-\lambda_2}.
\end{alignat*}
Then
\begin{equation}\label{fds}
H^*(\tau^n)= \theta(\tau^n) + A\frac{a_n+d_n}{c_n} = n\theta(\tau) + A\frac{\lambda_1^n+\lambda_2^n}{c_1(\lambda_1^n-\lambda_2^n)}(\lambda_1-\lambda_2).
\end{equation}
The last term in \e{fds} goes to $\pm A(\lambda_1-\lambda_2)/c_1$ as $n\to\infty$ and \eqref{lmh} follows.
\end{proof}

Equation \eqref{lmh} is not true when $\tau = \pm P_\ci^k$; the left side is $0$  and the right side is $-k A$.  For a further
discussion of these points, we refer to the proof of Lemma \ref{lem} below.

\section{Proofs of the main results}
The following initial lemmas establish basic properties of $H$ and $H^*$ which are required in our analysis.
\begin{lemma}\label{cor1}
Assume that $-I$ and $\g$ are in $\G$. Then $H(-\gamma)=H(\gamma)$ and $H^*(-\gamma)=H^*(\gamma)$. \end{lemma}
\begin{proof}
Write $\gamma =\left(\smallmatrix a & b \\ c & d \endsmallmatrix \right)$. For $c\neq 0$ we have
\begin{align*}
  H(-\g) & = -S(-\g)+A\frac{a+d}c-\frac{\sgn(-c)}4 \\
  & = -S(-I)-S(\g)-\omega(-I,\g)+A\frac{a+d}c+\frac{\sgn(c)}4\\
  & = \frac 12-S(\g)-\frac{1+\sgn(c)}2+A\frac{a+d}c+\frac{\sgn(c)}4 = H(\g).
\end{align*}
We used part i) of Theorem \ref{sa} and \eqref{phs2} for $k=0$.
The case when $c=0$ is similar.  The proof that $H^*(-\gamma)=H^*(\gamma)$ is  simpler, using only that
$\theta(-\gamma)=\theta(\gamma)$ which is from part ii) of Corollary \ref{thec}.
\end{proof}


\begin{lemma}\label{sym}
If $\iota(\Gamma)=\Gamma$ and $\g \in \G$ then
$
H(\iota(\gamma))=-H(\gamma)
$ and
$
H^*(\iota(\gamma))=-H^*(\gamma).
$
\end{lemma}
\begin{proof}
Use Definition 1.1, Theorem \ref{sa} part iii) and  Corollary \ref{thec} part iii) to verify these statements.
\end{proof}

\begin{lemma}\label{lem}    The functions  $H:\G \to \R$ and $H^{*}:\G \to \R$ are left and right $\Gamma_{\infty}$-invariant.
\end{lemma}
\begin{proof}
Recall that $\G_\ci$ is generated by $P_\ci$ and $-I$ if $-I \in \G$. If $-I \not\in \G$ then $\G_\ci$ is generated by $P_\ci$ alone. For  $\g$ in $\G$, parts i) and iv) of Theorem \ref{sa}, along with \e{comp2}, \e{phs}, show that
$$
S(P_\ci^k \g) = S(P_\ci^k)+S(\g)+\omega(P_\ci^k,\g) = kA+S(\g).
$$
Let $\gamma =\left(\smallmatrix a & b \\ c & d \endsmallmatrix \right)$. Then $P_\ci^k \g$ equals $\left(\smallmatrix a+kc & b+kd \\ c & d \endsmallmatrix \right)$. By definition, if $c\neq 0$,
\begin{align*}
 H(P_\ci^k \g) & = -S(P_\ci^k \g) + A\frac{a+kc+d}c-\frac{\sgn(c)}4 \\
 &  = -S(\g) + A\frac{a+d}c-\frac{\sgn(c)}4 = H(\g).
\end{align*}
If $c = 0$ then
\begin{align*}
 H(P_\ci^k \g)  & = -S(P_\ci^k \g) + A\frac{b+kd}d + \frac{\sgn(d)-1}4\\
  & = -S(\g) + A\frac{b}d + \frac{\sgn(d)-1}4= H(\g).
\end{align*}
Combining this with Lemma \ref{cor1} shows that $H$ is left-invariant under $\G_\ci$. The left-invariance of $H^*$ follows similarly
by employing
$$
\theta(P_\ci^k \g) = \theta(P_\ci^k)+\theta(\g) = -kA +\theta(\g)
$$
from parts i) and iv) of Corollary \ref{thec}. On the right, $\g P_\ci^k $ equals $\left(\smallmatrix a  & b+ka \\ c & d+kc \endsmallmatrix \right)$. By definition, if $c\neq0$,
\begin{align*}
H(\gamma P_\infty^k) &= -S(\gamma P_\infty^k) +A\frac{a+d+kc}{c} -\frac{\sgn(c)}{4}\\
&= -S(\gamma) +A\frac{a+d}{c} -\frac{\sgn(c)}{4} = H(\gamma).
\end{align*}
If $c=0$ then $a=d=\pm1$, as we have noted, and
\begin{align*}
H(\gamma P_\infty^k) &= -S(\gamma P^k_\infty) +A\frac{b+ka}{d} +\frac{\sgn(d)-1}{4}\\
& = -S(\gamma) +A\frac{b}{d} +\frac{\sgn(d)-1}{4}=H(\gamma).
\end{align*}
As previously, combining this with Lemma \ref{cor1} shows that $H$ is right-invariant under $\Gamma_\infty$. Similar calculations imply that $H^*$ is right-invariant as well.
\end{proof}

Analogous results to the above for  $H$ are contained in Theorem 2 of \cite{Bu}.

\begin{lemma}\label{lem-inv}   For all $\g \in \G$
\begin{equation}\label{inv}
    H( \g^{-1} ) = -H( \g) \quad \text{and} \quad    H^*( \g^{-1} ) = - H^*( \g).
\end{equation}
\end{lemma}
\begin{proof}
This is again a straightforward computation, using that $\omega(\g,\g^{-1}) = \rho(\g)$ in the notation of \e{rho}. In fact, for $c_\g\neq0$,
$$
H(\gamma)+H(\gamma^{-1}) = \omega(\gamma,\gamma^{-1})=\rho(\gamma)=0.
$$
For $c_\g = 0$, we know $a_\g=d_\g$ and hence
$$
H(\gamma)+H(\gamma^{-1}) = \rho(\gamma) + \frac{\sgn(d_\g)-1}{4} + \frac{\sgn(a_\g)-1}{4}=0.
$$
The result for $H^*$ is shown similarly.
\end{proof}

%


\begin{lemma}\label{Lem}
Let us write the product $\gamma \tau$ of any two matrices $\gamma$ and $\tau$ from $\Gamma$ as
$$
\left(\begin{matrix} a_{\gamma\tau} & b_{\gamma\tau}  \\ c_{\gamma\tau} & d_{\gamma\tau}  \end{matrix} \right)=
\left(\begin{matrix} a_{\gamma} & b_{\gamma}  \\ c_{\gamma} & d_{\gamma}  \end{matrix} \right)
\left(\begin{matrix} a_{\tau} & b_{\tau}  \\ c_{\tau} & d_{\tau}  \end{matrix} \right).
$$
Then, assuming that $c_{\gamma}c_{\tau}c_{\gamma\tau} \neq 0$, we have
$$
\frac{a_\gamma+d_\gamma}{c_\gamma} + \frac{a_\tau +d_\tau}{c_\tau} - \frac{a_{\gamma\tau}+d_{\gamma\tau}}{c_{\gamma\tau}}  =
\frac{c_\gamma}{c_\tau c_{\gamma\tau}} + \frac{c_\tau}{c_\gamma c_{\gamma\tau}} +\frac{c_{\gamma\tau}}{c_\gamma c_\tau},
$$
and
$$
\frac{a_{\gamma}+d_{\gamma}}{c_{\gamma}}  - \frac{a_{\gamma\tau}+d_{\gamma\tau}}{c_{\gamma\tau}}  =
 \frac{c_{\tau}}{c_{\gamma}c_{\gamma\tau}} - \frac{d_{\gamma\tau}}{c_{\gamma\tau}} + \frac{d_{\gamma}}{c_{\gamma}}.
$$
\end{lemma}

\begin{proof} The claim is proved by straightforward computation.
\end{proof}

\noindent
\textbf{Proof of the periodicity relations (Proposition 1.2).}
We already showed in Lemma \ref{lem} that $H$ and $H^*$ are left  $\Gamma_{\infty}$-invariant.

Assume that both $\gamma = \left(\smallmatrix a & b \\ c & d \endsmallmatrix \right)$
and $\eta = \left(\smallmatrix \alpha & \beta \\ c & d \endsmallmatrix \right)$ are in $\Gamma$.  Then
$$
\gamma \eta^{-1} = \left(\begin{matrix} ad-bc & \alpha b - a \beta \\ 0 & \alpha d - \beta c \end{matrix} \right)
= \left(\begin{matrix} 1 & \alpha b - a \beta \\ 0 & 1 \end{matrix}\right).
$$
This parabolic element fixes $\ci$ and necessarily equals $P_\ci^k$. Hence $\gamma = P_\ci^k \eta$ and by Lemma \ref{lem}, $H(\gamma)=H(\eta)$, and $H^*(\gamma)=H^*(\eta)$. This shows that $H(\gamma)$ and $H^{*}(\gamma)$ depend solely
on the lower row of $\gamma$, and therefore we may use the notation \eqref{notnx}.

We have $\gamma P_\infty = (\smallmatrix a&b+a\\ c& d+c\endsmallmatrix)$. Then by the right $\Gamma_\infty$-invariance of $H$, 
$$
H(d+c,c)=H(\gamma P_\infty) = H(\gamma) =  H(d,c).
$$
The same is true for $H^*$.

\vskip .10in
\noindent
\textbf{Proof of the three-term relations (Theorem \ref{tr-thm}, Corollary \ref{3term}).}
By Definition 1.1, one has that
\begin{equation*}
  H^*(\gamma\tau) - H^*(\gamma) - H^*(\tau) = \theta(\gamma\tau) -\theta(\gamma)-\theta(\tau) -A\left(\frac{a_\gamma+d_\gamma}{c_\gamma}+
  \frac{a_\tau+d_\tau}{c_\tau}-\frac{a_{\gamma\tau}+d_{\gamma\tau}}{c_{\gamma\tau}}\right).
\end{equation*}
Applying part i) of Corollary \ref{thec} and Lemma \ref{Lem} on the right gives \e{hhx2}. The analogous relation for $H$ in \e{hhx1} is proved in the same way, using part i) of Theorem \ref{sa} and \eqref{phs3}.
This  completes the proof of Theorem \ref{tr-thm}.

Let $\gamma_1=\gamma$, $\gamma_2=\tau$, and $\gamma_3=(\gamma\tau)^{-1}$ in Theorem \ref{tr-thm}. By Lemma \ref{lem-inv}, we have $H(\g_3)=-H(\g\tau)$ and $H^*(\g_3)=-H^*(\g\tau)$. Corollary \ref{3term} follows upon noting that $c_3=-c_{\g\tau}$.

\vskip .10in
\noindent
\textbf{Proof of the two-term relations (Theorem \ref{2trm}).}
With  Definition 1.1,
$$
H(\gamma)-H(\gamma\tau)  = -S(\g) +S(\g\tau)+A\left(\frac{a_\gamma+d_\gamma}{c_\gamma} -\frac{a_{\gamma\tau}+d_{\gamma\tau}}{c_{\gamma\tau}}\right) - \frac14 \sgn(c_\gamma ) + \frac14 \sgn( c_{\gamma\tau}).
$$
The relation \eqref{recip2a} is then established with Theorem \ref{sa} part i), \eqref{phs3},  Lemma \ref{Lem} and using the notation \eqref{notnx}.
Similarly,
$$
H^*(\gamma) - H^*(\gamma\tau) = \theta(\g ) - \theta(\g \tau) + A\left(\frac{a_\gamma+d_\gamma}{c_\gamma} -\frac{a_{\gamma\tau}+d_{\gamma\tau}}{c_{\gamma\tau}}\right)
$$
and then \eqref{recip3}  follows from Corollary \ref{thec} part i) and Lemma \ref{Lem}.

Suppose that $\tau$ is parabolic or elliptic. We have seen that this means $\s{\tau}{f}=0$.
Also, if $\tau$ does not fix a cusp equivalent to $\infty \bmod \G$, then $\theta(\tau)=0$ by Corollary \ref{thec} parts iv) and v).  Therefore \eqref{recip3} reduces to \eqref{recip2b} in this case.


\vskip .10in
\noindent
\textbf{Proof of the density result (Theorem 1.3).}
By Corollary \ref{iteration}, $\theta$ depends only on its values on a fixed set of generators for $\Gamma$ and their commutators.  More precisely,
$\theta(\gamma)$ is a $\tfrac12 \Z$-linear combination of this finite set of values. Recall that it is shown in \cite[Sect. 6.4]{JO'SS1} that
$\theta$ is rational on the set of generators and their commutators.  Let $L$ be the least common multiple of all the denominators.
Then, $2L\cdot \theta(\gamma) \in \Z$ for all $\gamma\in\Gamma$.

On the other hand, by the Gauss--Bonnet formula,
\begin{equation}\label{gbon}
  A = g-1 + \frac{n}{2} + \frac12\sum_{j=1}^\ell \left(1-\frac{1}{m_j}\right)
\end{equation}
where $g$ is the genus of $\Gamma$, $n$ is the number of inequivalent cusps and the numbers $m_j$ give the orders of the generating elliptic elements.
There can only be finitely many elliptic generators.  Let $L'$ be the least common multiple of $m_1,\dots, m_\ell$. It then follows
that $A\in \tfrac{1}{2L'}\Z$. Set $t$ to be the positive integer $4LL'$. Let $e(x):=e^{2\pi ix}$. Then for any $m,n\in\N$,
\begin{align*}
\sum_{c\leq X}\sum_{0\leq d<c} e(m\cdot \tfrac{d}{c} + n\cdot t H^*(d,c))  &=
\sum_{c\leq X} \sum_{0\leq d<c} e(m\cdot \tfrac{d}{c}+n\cdot tA \tfrac{a+d}{c}) \\
& = \sum_{c\leq X}\sum_{0\leq d<c} e(ntA \tfrac{a}{c} + (m+ntA)\tfrac{d}{c}).
\end{align*}
By the preceding discussion, both $M:=ntA$ and $N:=m+M$ are positive integers.  Therefore, the series
$$
S(M,N,c) = \sum_{0\leq d<c} e(\tfrac{Ma+Nd}{c})
$$
defines a Selberg--Kloosterman sum.
Let $\delta_0=1$ if $M=N=0$ and zero otherwise.  Theorem 4 in \cite{Go} establishes the existence of a real number $\kappa>0$ such that
$$
\sum_{c\leq X} S(M,N,c) = \delta_0 \frac{X^2}{\pi V_\Gamma} + O(X^{2-\kappa})
\,\,\,\,\,\text{\rm as $X \rightarrow \infty$.}
$$
Theorem 1.3 then follows from the classical Weyl criterion; see, for example, \cite{KN74}.

\section{Further transformation relations}

The two-term relations in Theorem \ref{2trm} yield additional arithmetic properties of the generalized Dedekind sums $H$ and $H^*$,
some of which we now state.

\begin{cor}\label{cor2}
Assume that $w:=\left(\begin{smallmatrix} 0&-1/\lambda\\ \lambda&0\end{smallmatrix}\right)\in\Gamma$ for $\lambda>0$. Let $\gamma =
\left(\begin{smallmatrix} *&*\\c&d\end{smallmatrix}\right)\in\Gamma$ be such that $c d\neq 0$.  Then
\begin{align}\label{der1}
  H(d,c) - H(-c/\lambda,d \lambda) & = A\left(\frac{d}{c}+\frac{c}{d \lambda^2} +\frac{1}{cd}\right) -\frac14 \sgn(cd), \\
  H^*(d,c) - H^*(-c/\lambda,d \lambda) & = A\left(\frac{d}{c}+\frac{c}{d \lambda^2} +\frac{1}{cd}\right). \label{der2}
\end{align}
Moreover, if $\iota(\Gamma)=\Gamma$ then $-H(-c/\lambda,d \lambda)=H(c/\lambda,d \lambda)$ and $-H^*(-c/\lambda,d \lambda)=H^*(c/\lambda,d \lambda)$.
\end{cor}
\begin{proof}
The relation \e{recip2b} for $\tau = w$ implies \eqref{der2} directly.  Also \e{recip2a} implies
\begin{equation*}
  H(d,c) -H(-c/\lambda,d\lambda)  = A\left(\frac{d}{c}+\frac{c}{d \lambda^2} +\frac{1}{cd}\right) +S(w)+\frac14 \big(1-\sgn(cd) \big).
\end{equation*}
 Since  $S(w^2)=S(-I)=-1/2$, and $S(w^2)=2S(w)+\omega(w,w)=2S(w)$,  we can conclude that $S(w)=-1/4$.
For the last statement,  Lemmas \ref{cor1} and \ref{sym} combine to show
$H(\gamma) = -H(\iota(\gamma)) = - H(-\iota(\gamma))$. Note that $-I = w^2 \in \G$.
The same properties hold for $H^*$.
\end{proof}

When zero is a cusp for the group $\G$ there is a particularly elegant form of Theorem \ref{2trm}.

\begin{cor} \label{rtra}
Suppose that $P_{0} := \left(\smallmatrix 1 & 0 \\ s & 1 \endsmallmatrix \right) \in \G$ for  $s >0$, and assume that the cusps $0$ and $\infty$
are not $\G$ equivalent. Then for $k \in \Z$ and $\gamma =
\left(\begin{smallmatrix} *&*\\c&d\end{smallmatrix}\right)\in\Gamma$ with $c(c+ksd) \neq 0$,
 the identity
\begin{equation}\label{hid}
  H^*(d,c)- H^*(d,c+ksd)= A\frac{ks(d^2+1)}{c(c+ksd)}
\end{equation}
holds. Furthermore, if $c(c+ksd) > 0$, then
\begin{equation}\label{hid2}
  H(d,c)- H(d,c+ksd)= A\frac{ks(d^2+1)}{c(c+ksd)}.
\end{equation}
\end{cor}
\begin{proof}
Both \eqref{hid} and \eqref{hid2} are true for $k=0$ so we may assume $k \neq 0$. The first identity \eqref{hid} follows by taking $\tau=P_0^k$ in \eqref{recip2b}. By \eqref{recip2a},
$$
H(d,c)- H(d,c+ksd)= A\frac{ks(d^2+1)}{c(c+ksd)}+S(P_0^k) +\frac 14 \sgn(k)\big( 1-\sgn(c(c+ksd))\big).
$$
With Theorem \ref{sa} part iv) and \eqref{comp2} we have $S(P_0^k)=0$ and the second identity \eqref{hid2}  follows.
\end{proof}

Numerous further relations and identities for $H$ and $H^*$ are possible. For example, suppose $\iota(\G)=\G$ and
$\gamma =
\left(\begin{smallmatrix} a& *\\c&d\end{smallmatrix}\right)\in\Gamma$ with $c \neq 0$. Then by Lemmas \ref{sym} and \ref{lem-inv} one finds that $H(d,c)= H(a,c)$ and $H^*(d,c)= H^*(a,c)$.

\section{Examples}

In this section we present  transformation laws for the generalized Dedekind sums $H$ and $H^*$ associated to  Hecke triangle groups,
 Hecke congruence groups and  Helling/moonshine type groups. We include the underlying group  in the notation for clarity.

\subsection{Hecke triangle groups}
Let $\lambda_q:=2\cos(\pi/q)$. The Hecke triangle group $G_q$ is generated by
$\left(\begin{smallmatrix} 0 & -1  \\ 1 & 0  \end{smallmatrix} \right)$ and $\left(\begin{smallmatrix} 1 & \lambda_q  \\ 0 & 1  \end{smallmatrix} \right)$
for integers $q\geq 3$.
For each group $G_{q}$, the corresponding Riemann surface has genus zero with one cusp, two elliptic points and  volume  $\pi(q-2)/q$. Except for $q=3, 4, 6$ and
$\infty$, $G_{q}$ is non-arithmetic.

We conjugate $G_q$  into our desired form to obtain
$$
\hat{G_q}:= \sigma^{-1} G_q \sigma = \left\langle \left(\begin{matrix} 0 & -1/\lambda_q  \\ \lambda_q & 0  \end{matrix}\right), \left(\begin{matrix}
1 & 1  \\ 0 & 1  \end{matrix} \right) \right\rangle
\quad \text{for} \quad \sigma = \left(\begin{matrix} \lambda_q^{1/2} & 0  \\ 0 & \lambda_q^{-1/2}  \end{matrix} \right).
$$
The generalized Dedekind sums  $H^*$ do not exist for these groups as there are no weight two cusp forms. For $H_{\hat{G_q}}$ associated to $\hat{G_q}$,
Corollary \ref{cor2} implies that, with any $
\left(\begin{smallmatrix} * & *\\c&d\end{smallmatrix}\right)\in \hat{G_q}$  and $cd>0$,
\begin{equation}\label{rec}
  H_{\hat{G_q}}(d,c) + H_{\hat{G_q}}(c/\lambda_q,d \lambda_q)  = \frac{q-2}{4q}\left(\frac{d}{c}+\frac{c}{d \lambda_q^2} +\frac{1}{cd}\right) -\frac14.
\end{equation}
This is equivalent to the reciprocity from \cite{Br}, p. 12. We used that $\iota(\hat{G_q}) = \hat{G_q}$ since $\iota$ maps the generators to their inverses and satisfies $\iota(\g \tau)=\iota(\g)\iota(\tau)$.

When $q=3$ we have $\lambda_3=1$ and $G_3=\hat{G_3}=\text{\rm SL}(2, \mathbb{Z})$. Then \eqref{rec} becomes
$$
H_{\text{\rm SL}(2, \mathbb{Z})}(c,d) + H_{\text{\rm SL}(2, \mathbb{Z})}(d,c) = \frac{1}{12}\left(\frac{d}{c} + \frac{c}{d} +\frac{1}{cd}\right) - \frac{1}{4},
$$
which is the reciprocity law for classical Dedekind sums.

\subsection{Hecke congruence groups} \label{sec:Hecke congruence}
Define $\Gamma_0(N)$ as the subgroup of $\SL(2,\Z)$ with bottom left matrix elements divisible by the level $N$. In this subsection we take $\Gamma$ to be one of these Hecke congruence groups, with subscript $N$ in our notation indicating the level.

With the work of Vassileva \cite{Va}, it can be seen that $H_N$ may be expressed with a finite sum, as in the classical case.
Define the products $\alpha_N:=\prod_{p\mid N} 1/(1-p^{-1})$ and  $\beta_N:=\prod_{p\mid N}(1-p^{-2})/(1-p^{-1})$ where $p$ is prime,
and let $\mu$ denote the M\"obius function. In \cite[Thm. 4.1.3, Prop. 4.1.1]{Va} it is shown that
for
$\g= \left(\smallmatrix a & b \\ c & d \endsmallmatrix \right) \in \Gamma_0(N)$,
\begin{equation} \label{vas}
 S_N(\g)=
 \beta_N\frac{N(a+d)}{12c}   - \frac{c}{4|c|} - \alpha_N \frac{c}{|c|} \sum_{v\mid N}
 \frac{\mu(v)}{v}s\left( d, \frac{v|c|}{N}\right) \quad  \text{  if } \quad c\neq 0
\end{equation}
and
\begin{equation*}
  S_N(\g)= \beta_N \frac{Nb}{12d} + \frac{d/|d|-1}{4} \quad  \text{  if } \quad  c= 0.
\end{equation*}
Since $N\beta_N$ is the index of the group $\Gamma_0(N)$ in $\mathrm{SL}(2,\Z)$, one has $A_N=N\beta_N/12$. Hence
\begin{equation}\label{hcongruence}
H_N(d,c)=\alpha_N \frac{c}{|c|} \sum_{v\mid N} \frac{\mu(v)}{v}s\left( d, \frac{v|c|}{N}\right) \quad  \text{  if } \quad c\neq 0
\end{equation}
and $H_N(d,c)=0$ if $c=0$.

For each parabolic generator of $\Gamma_0(N)$ we obtain a periodicity relation for $H_N$. The generator $P_\ci$ gives the
usual periodicity $H_N(d+c,c)=H_N(d,c)$ from Proposition \ref{period}. This also follows from \e{hcongruence} and \e{perio}.
Next, if we assume that $N\geq 2$, then the cusps $0$ and $\infty$ are not $\Gamma_0(N)$ equivalent. Let $c$, $d$ and $k$ be positive
integers with $c$ and $d$ relatively prime and $N | c$. Corollary \ref{rtra} may be applied with $s=N$ to obtain the periodicity
\begin{equation}\label{hid2q}
  H_N(d,c)- H_N(d,c+kNd)= A_N\frac{kN(d^2+1)}{c(c+kNd)}.
\end{equation}
 Inserting \eqref{hcongruence} into \e{hid2q} then yields
 the summation formula
\begin{equation}\label{rty}
\sum_{v\mid N} \frac{\mu(v)}{v}\left( s\left( d, \frac{v c}{N}\right) - s\left( d, \frac{v c}{N} + v k d\right) \right) =
\frac{N^2 k(d^2+1)}{12c(c+kNd)}  \prod_{p\mid N} \left(1-p^{-2} \right).
\end{equation}
We are unaware of this identity involving the classical Dedekind sum \eqref{dede} and the M\"obius function
appearing elsewhere in the literature. However, a direct proof of \e{rty} may be given by applying the reciprocity law \e{recip} and then the periodicity relation \e{perio} to the second Dedekind sum. With another application of reciprocity the Dedekind sums disappear, and we are left with an identity that follows from the properties of $\mu$.

The groups $\Gamma_0(N)$ are genus zero  for finitely many $N$. Outside these levels, $H_N^*$ may be constructed as there are weight two cusp forms.
It is not known if $H_N^*$ may be expressed as a finite sum, similar to \e{hcongruence}. Indirect evidence that such sums exist is given by the fact that the
values of $H_N^*$ are rational in all the cases examined so far. In particular, it follows from Theorem 4 of \cite{JO'SS1} that $H_N^*$ is always
rational for $N\in \{11,14,15,17,19,20,21,24,27,32,36,49\}$. These are the levels where $\Gamma_0(N)$ has genus one and so $H_N^*$ is independent of the
choice of nonzero cusp form.

\subsection{The group $\Gamma_{0}(11)$} \label{gam11}
Let us focus on the Hecke congruence group of level $N=11$, the simplest case where both $H_N=H_{11}$ and $H_N^*=H_{11}^*$ are not trivial.
The group $\G=\G_0(11)$ is generated by the parabolic and hyperbolic elements
\begin{equation}\label{g11}
  P_\ci := \left(\begin{matrix} 1 & 1 \\ 0 & 1\end{matrix}\right)
,\quad
P_0 := \left(\begin{matrix} 1 & 0 \\ 11 & 1\end{matrix}\right),\quad
Q :=\left(\begin{matrix} -7 & -1 \\ 22 & 3\end{matrix}\right),
\quad
R :=\left(\begin{matrix} 4 & 1 \\ -33 & -8\end{matrix}\right),
\end{equation}
along with $-I$, and they satisfy the relation $QRQ^{-1}R^{-1}P_0^{-1}P_\ci=I$. Thus, $H_{11}$ and $H_{11}^*$ have the periodicity relations corresponding to $P_\ci$, $P_0$ from Proposition \ref{period} and Corollary \ref{rtra}. There are also hyperbolic periodicity relations corresponding to $Q$ and $R$, from Theorem \ref{2trm}, which we describe next. In the notation of Section \ref{sec:Hecke congruence} above, $A_{11}=11\beta_{11}/12=1$.

By the computations in \cite[Sect. 7.1]{JO'SS1}, with $A$ and $B$ replaced notationally by $Q$ and $R$, we have that
\begin{equation*}
  S_{11}(P_\ci)=1, \quad S_{11}(P_0)=0, \quad S_{11}(Q)=-\frac 25, \quad S_{11}(R)=\frac 25,
\end{equation*}
and, by extending that work slightly, the evaluations
\begin{equation}\label{valss}
  \theta_{11}(P_\ci)=-1, \quad \theta_{11}(P_0)=0, \quad \theta_{11}(Q)=\frac 3{10}, \quad \theta_{11}(R)=-\frac 3{10}.
\end{equation}
We may take $f(z)$ to be $\eta(z)^2\eta(11z)^2$. If $\g= \left(\smallmatrix * & * \\ c & d \endsmallmatrix \right) \in \G$ then
\begin{equation*}
  \s{\g}{f} = - \s{\g^{-1}}{f} = -2\pi i \int_{\ci}^{\g^{-1}\ci} f(w)\, dw = 2\pi i \int^{\ci}_{-d/c} f(w)\, dw.
\end{equation*}
For all $c$, $d \in \Z$ that are relatively prime, with $11|c$ and $c(22d-7c)\neq 0$, the hyperbolic periodicity relations associated to $Q$ are
\begin{equation*}
  H_{11}(d,c)+H_{11}(c-3d,-7c+22d)  = \left(\frac dc +\frac{c-3d}{22d-7c}+\frac{22}{c(22d-7c)}\right) -\frac 25 + \frac 14(1-\sgn(c(22d-7c))),
\end{equation*}
and
\begin{multline} \label{mods}
  H_{11}^*(d,c)+H_{11}^*(c-3d,-7c+22d)  = \left(\frac dc +\frac{c-3d}{22d-7c}+\frac{22}{c(22d-7c)}\right)-\frac 3{10} \\
  -\frac{1}{2\|f\|^2} \Im\left(\int^{\ci}_{-d/c} f(w)\, dw \cdot \overline{\int^{\ci}_{-3/22} f(w)\, dw } \right).
\end{multline}
The relations for $R$ are similar, and we may also give the transforms corresponding to the inverses of $Q$ and $R$.

Using the above transformations, $H_{11}(d,c)$ and $H_{11}^*(d,c)$  can be computed, for any pair $d$ and $c$ corresponding to $\g= \left(\smallmatrix * & * \\ c & d \endsmallmatrix \right) \in \Gamma_0(11)$ as follows.  Write $\g$ as a word in the generators \e{g11}. At each step apply the  transform corresponding to the inverse of
the rightmost generator appearing in the word. This eventually reduces $\g$ to $I$ or $-I$. For $H_{11}^*$ this procedure will involve the modular symbols in \e{mods}.

A more efficient method to evaluate $H_{11}^*(\g)$ is to apply Corollary \ref{iteration} to the word representing $\g$. This gives $\theta_{11}(\g)$ in terms of $\theta_{11}$ evaluated on generators and commutators of pairs of generators. It  follows easily from Lemma \ref{prop-theta} that
$$
\theta_{11}([\gamma,\tau]) = \theta_{11}(\g\tau)-\theta_{11}(\tau\g)
$$
and so the problem reduces further. It is shown in \cite[Sect. 7.1]{JO'SS1}
how to compute $\theta_{11}$ of products of pairs of generators. Using the values \e{valss} then completes the calculation.

\subsection{The groups $\Gamma_{0}(N)^{+}$}
Assume that $N$ is a square-free positive integer with $r$ factors. In this subsection we let $\G$ be the Helling or moonshine type group of level $N$, which is
\begin{equation*}
  \Gamma_0(N)^+ := \left\{\frac 1{\sqrt{e}}\m{a}{b}{c}{d} \in
    \SL(2,\R):\ a,b,c,d,e\in\Z, \ e\mid N,\ e\mid a,
    \ e\mid d,\ N\mid c \right\}.
\end{equation*}
Then this $\G$ has only one cusp, at $\infty$, with $\G_\ci$ generated by $P_\ci$ and $-I$. The constant $A=A_{N,+}$ in this case
is $\sigma(N)/(12\cdot 2^{r})$ where $\sigma(N)$ indicates the sum of all divisors of $N$.  See \cite{JST16} and its included references  for
further information on $ \Gamma_0(N)^+$.

The generalized Dedekind sum $H_{N,+}$ associated to $\Gamma_{0}(N)^{+}$  satisfies the periodicity relation of Proposition \ref{period}. Moreover, if $p$ divides $N$ then
$\left(\smallmatrix 0 & -1/\sqrt{p} \\ \sqrt{p} & 0 \endsmallmatrix \right) \in \Gamma_{0}(N)^{+}$.  Let $
\left(\begin{smallmatrix} *&*\\c&d\end{smallmatrix}\right)\in\Gamma_{0}(N)^{+}$ with $cd>0$.  By Corollary \ref{cor2} we
obtain the  reciprocity relation
\begin{equation}\label{hh0}
  H_{N,+}(d,c) + H_{N,+}(c/\sqrt{p},d\sqrt{p}) = \frac{\sigma(N)}{12\cdot 2^{r}}\left(\frac{d}{c} + \frac{c}{pd} + \frac{1}{cd}\right) - \frac{1}{4}.
\end{equation}

Further elliptic elements will result in additional reciprocity laws.  For example, when $N=37$, there are three more elliptic elements of $\Gamma_0(37)^+$:
$$
\epsilon_{1}:= \left(\begin{matrix} 6 & -1 \\ 37 & -6\end{matrix}\right)
,\quad
\epsilon_{2}:=\frac{1}{\sqrt{37}}\left(\begin{matrix} 37 & -19 \\ 74 & -37\end{matrix}\right)
\quad \text{\rm and}\quad
\epsilon_{3}:=\left(\begin{matrix} 11 & -3 \\ 37 & -10\end{matrix}\right).
$$
The first two are of order four and the third one is of order six (see \cite[Sect. 7.2]{JO'SS1}).
Computing with Theorem \ref{sa} part v) yields that $S_{37,+}(\epsilon_{1})=S_{37,+}(\epsilon_{2})=-1/4$ and $S_{37,+}(\epsilon_{3})=-1/6$.
Then, for $c,d \in\Z$ or $c,d\in\Z\sqrt{37}$, $cd>0$ such that $
\left(\begin{smallmatrix} *&*\\c&d\end{smallmatrix}\right)\in\Gamma_0(37)^+$, applying Theorem \ref{2trm} with $\tau=\epsilon_1$, $\epsilon_2$ and $\epsilon_3$ (using Lemmas \ref{cor1}, \ref{sym} to adjust the signs) gives the following formulas:
\begin{align}
  H_{37,+}(d,c) + H_{37,+}(c + 6d,6c + 37d) & = \frac{19}{12}\left(\frac{d}{c} + \frac{c+6d}{6c+37d} + \frac{37}{c(6c+37d)}\right) - \frac{1}{4}, \label{hh1} \\
  H_{37,+}(d,c) + H_{37,+}((19c + 37d)/\sqrt{37},\sqrt{37}(c + 2d)) & = \frac{19}{12}\left(\frac{d}{c} +
\frac{19c+37d}{37(c + 2d)} + \frac{2}{c(c + 2d)}\right)- \frac{1}{4}, \label{hh2} \\
H_{37,+}(d,c) + H_{37,+}(3c + 10d,11c + 37d) & = \frac{19}{12}\left(\frac{d}{c} + \frac{3c + 10d}{11c + 37d} + \frac{37}{c(11c + 37d)}\right)-\frac{1}{6}. \label{hh3}
\end{align}

Recall that  $\Gamma_{0}(37)^+$ has genus one, and hence the  Dedekind sum  $H_{37,+}^{*}$ exists for this group.  It is uniquely defined as $\dim S_2(\G)=1$. Using \e{recip2b}, one obtains reciprocity formulas
for $H_{37,+}^{*}$ that are almost the same as those for $H_{37,+}$ above;  replace $H_{37,+}$ by $H_{37,+}^{*}$ in \eqref{hh0}--\eqref{hh3} and remove the fractions $1/4$ and $1/6$ on the right.

The group $\Gamma_{0}^+(37)$ is generated by $P_\infty$, the elliptic elements $\left(\smallmatrix 0 & -1/\sqrt{37} \\ \sqrt{37} & 0 \endsmallmatrix \right)$, $\epsilon_1$, $\epsilon_2$, $\epsilon_3$ and the hyperbolic elements
$$
L:= \frac{1}{\sqrt{37}}\left(\begin{matrix} 148 & -89 \\ 185 & -111\end{matrix}\right)
,\quad
M:=\left(\begin{matrix} 20 & -13 \\ 37 & -24\end{matrix}\right).
$$
Hyperbolic periodicity relations associated to $L$, as an example, can be easily computed by combining Theorem \ref{2trm} with computations from \cite[Sect. 7.2]{JO'SS1}.

The compactified quotient $X_{0}(37)^{+}:=\overline{\G_0(37)^+}\setminus \H$ is isomorphic to the genus one curve $y^2 =4x^3-4x+1$ over $\C$,
where $\overline{\G_0(37)^+}$ is the projection of $\Gamma_{0}(37^+)$ onto $\mathrm{PSL}(2,\R)$. Hence, we can take $f(z)$ to be the weight
two cusp form on $\Gamma_0(37)^+$ related to the isomorphism between $\G\setminus \H \cup \{\ci\}$ and the algebraic curve $y^2 =4x^3-4x+1$
sending $\ci$ to $0$ in such a way that the pull-back of the canonical differential $dx/dy$ is $-2\pi i f(z)dz$. The Petersson norm of $f$
is then $||f||=\omega_1\omega_2/(4\pi^2 i)$, where $\omega_1$ and $\omega_2$ are the real and complex periods of the curve $y^2 =4x^3-4x+1$.

By the computations in \cite[Sect. 7.2]{JO'SS1}, one has that $\theta_{37,+}(L)= -\frac{19}{24}$. Hence, for all $\left(\begin{smallmatrix} *&*\\c&d\end{smallmatrix}\right)\in\Gamma_0(37)^+$ such that $c(4c+5d)\neq0$ we find
\begin{multline*}
 H_{37,+}^*(d,c)+H_{37,+}^*((89c+111d)/\sqrt{37},(4c+5d)\sqrt{37})  = \frac{19}{12}\left(\frac dc +\frac{89c+111d}{37(4c+5d)}+\frac{5}{c(4c+5d)}\right)  +\frac {19}{24}\\
 -\frac{19}{24\|f\|^2} \Im\left(\int^{\ci}_{-d/c} f(w)\, dw \cdot \overline{\int^{\ci}_{3/5} f(w)\, dw } \right).
\end{multline*}
A similar expression can be easily deduced for the sum $H_{37,+}$.

{\small

}

\vskip 5mm

{\footnotesize
\textsc{Department of Mathematics, Rutgers University, Piscataway, NJ 08854, USA}

\textit{E-mail address:} \texttt{claire.burrin@rutgers.edu}

\vskip 3mm

\textsc{Department of Mathematics,
 The City College of New York,
 New York, NY 10031, USA
 }

\textit{E-mail address:} \texttt{jjorgenson@mindspring.com}

\vskip 3mm

\textsc{Department of Mathematics, The CUNY Graduate Center, New
   York, NY 10016, USA}

\textit{E-mail address:} \texttt{cosullivan@gc.cuny.edu}

\vskip 3mm

\textsc{Department of Mathematics, University of Sarajevo, 71 000 Sarajevo, Bosnia and Herzegovina}

\textit{E-mail address:} \texttt{lejlas@pmf.unsa.ba}
}

\begin{thebibliography}{99}


\bibitem{Br}
R. Bruggeman, \textit{Dedekind sums for Hecke groups}, Acta Arith. \textbf{71} (1995), 11 - 46.

\bibitem{Bu}
C. Burrin, \textit{Generalized Dedekind sums and equidistribution mod $1$},
J. Number theory \textbf{172} (2017), 270--286.
\bibitem{Bu2}
C. Burrin, \textit{Reciprocity of Dedekind sums and the Euler class}, Proc.~Amer.~Math.~Soc.~\textbf{146}, (2018), 1367-1376.

\bibitem{CHO} G. Chinta, I. Horozov, C. O'Sullivan,
\textit{Noncommutative modular symbols and Eisenstein series},
to appear in  \it Automorphic Forms and Related Topics, \rm
AMS Contemporary Mathematics, eds. S. Anni, J. Jorgenson, L. Smajlovi\'c and L Walling.

\bibitem{Di}
U. Dieter, \textit{Beziehungen zwischen Dedekindschen Summen}. Abh. Math. Semin. Univ. Hamburg \textbf{21} (1957), 109--125.

\bibitem{DIT}
W. Duke, \"O. Imamo\=glu, and \'A. T\'oth,
\textit{Kronecker's first limit formula, revisited}.
Res. Math. Sci. \textbf{5} (2018), Paper No. 20, 21 pp.
'

\bibitem{Gn1}
L. Goldstein, \textit{Dedekind sums for a Fuchsian group. I},
 Nagoya Math. J. \textbf{50} (1973), 21 - 47.

\bibitem{Gn2}
L. Goldstein, \textit{Errata for "Dedekind sums for a Fuchsian group. I'' (Nagoya Math. J. 50 (1973), 21 - 47)},
Nagoya Math. J. \textbf{53} (1974), 235 - 237.


 \bibitem{Gn3}
L. Goldstein, \textit{Dedekind sums for a Fuchsian group. II},
 Nagoya Math. J. \textbf{53} (1974), 171 - 187.

\bibitem{Go}
A. Good, Local analysis of Selberg's trace formula. Lecture Notes in Mathematics, vol.~1040, Springer-Verlag, 1983.

\bibitem{Iw2} H. Iwaniec, \textit{Topics in classical automorphic forms}, Graduate Studies in Mathematics \textbf{17}, American Mathematical Society, Providence, RI, 1997.


\bibitem{JO'S05} J. Jorgenson, C. O'Sullivan, \textit{Convolution Dirichlet series and a
Kronecker limit formula for second-order Eisenstein series},
Nagoya Math. Journal \textbf{179}, (2005), 47--102.

\bibitem{JO'SS1} J. Jorgenson, C. O'Sullivan, L. Smajlovi\'c,
\textit{Modular Dedekind symbols associated to Fuchsian groups and higher-order Eisenstein series},
submitted for publication.

\bibitem{JST16}
J. Jorgenson, L. Smajlovi\'c, H. Then,
\textit{Kronecker's limit formula, holomorphic modular functions and $q$-expansions
on certain arithmetic groups},  Experimental Math. \textbf{25}, (2016), 295--320.

\bibitem{KN74} L. Kuipers and H. Niederreiter, \textit{Uniform distribution of sequences},  Pure and Applied Mathematics.
Wiley-Interscience [John Wiley $\&$ Sons], New York--London--Sydney, 1974, xiv+390.

\bibitem{La95}
S. Lang, \textit{Introduction to modular forms},
Grundlehren der Mathematischen Wissenschaften \textbf{222} Springer-Verlag, Berlin, 1995.

\bibitem{My}
G. Myerson, \textit{Dedekind sums and uniform distribution}, J. Number Theory {\bf 28} (1988), 233-239.

\bibitem{OS00}
C. O'Sullivan, \textit{Properties of Eisenstein series formed with modular symbols},
J. Reine Angew. Math. \textbf{518} (2000), 163-–186.

\bibitem{Pet} H. Petersson, \textit{Zur analytischen Theorie der Grenzkreisgruppen I}, Math. Ann.
\textbf{115} (1938), 23 - 67.

\bibitem{PR}
Y. Petridis, M. Risager,
\textit{Arithmetic statistics of modular symbols},
Invent. Math. \textbf{212} (2018), 997-–1053.


\bibitem{RG72}
H. Rademacher, E. Grosswald, \textit{Dedekind sums}, The Carus Mathematical Monographs, No.
\textbf{16}. The Mathematical Association of America, Washington, D.C., (1972).

\bibitem{S}
G. Shimura,
\newblock {\em Introduction to the arithmetic theory of automorphic functions},
  volume~11 of {\em Publications of the Mathematical Society of Japan}.
\newblock Princeton University Press, Princeton, NJ, 1994.
\newblock Reprint of the 1971 original, Kan{\^o} Memorial Lectures, 1.

\bibitem{Va}
I. Vassileva, \textit{Dedekind eta function, Kronecker limit formula and Dedekind sum for the Hecke group},
Thesis (Ph.D.) University of Massachusetts Amherst. 1996. 81 pp. ISBN: 978-0591-04694-6.


\end{thebibliography}
\end{document}